\documentclass[a4paper, 12pt]{article}

\usepackage{authblk}
\usepackage{comment}
\usepackage{amssymb}
\usepackage{amsthm}
\usepackage{amscd}
\usepackage{graphicx}
\usepackage{titlesec}
\titleformat{\section}{\bfseries}{\thesection .}{0.5em}{}
\titleformat{\subsection}{\itshape}{\thesubsection .}{0.5em}{}
\usepackage[top=25.4mm, bottom=25.4mm, left=31.7mm, right=32.2mm]{geometry}
\usepackage[tbtags]{amsmath}
\usepackage{booktabs}
\allowdisplaybreaks  

\usepackage{caption}
\begin{document}

\newtheorem{example}{Example}             
\newtheorem{algorithm}{Algorithm}
\newtheorem{thm}{Theorem}  
\newtheorem{definition}{Definition}
\newtheorem{axiom}{Axiom}
\newtheorem{prop}{Property}
\newtheorem{pro}{Proposition}
\newtheorem{lem}{Lemma}
\newtheorem{cor}{Corollary}
\newtheorem{rmk}{Remark}
\newtheorem{condition}{Condition}
\newtheorem{conclusion}{Conclusion}
\newtheorem{assumption}{Assumption}

\title{\Large A two-queue Markovian polling system with two priority levels in the first queue}
\author{
{\normalsize Chu Yuqing\thanks{Corresponding author\\
\textit{Email address}: chuyuqing@csu.edu.cn (Chu Yuqing), math\_lzm@csu.edu.cn (Liu Zaiming), wujinbiao@csu.edu.cn (Wu Jinbiao)}, Liu Zaiming, Wu Jinbiao}\\
\textit{\small Department of Mathematics and Statistics, Central
South University, Changsha, Hunan 410083, PR China}}
\date{}
\maketitle

\renewcommand{\baselinestretch}{1.2} 
\large\normalsize
\noindent \rule[0.5pt]{14.5cm}{0.6pt}\\
\noindent
\textbf{Abstract}\\
In this paper, we deal with a two-queue polling system attended by a single server. The server visits the queues according to a Markovian routing mechnism. There are two-class customers in the first queue. Customers of each queue are served in the exhaustive discipline. For this model, we study the joint number of customers, the cycle time and the waiting times. we are also devoted to deriving the exact asymptotics for the scaled delay in the heavy-traffic scenario. In addition, the scaled delays with large switch-over times are discussed for the first time in the Markovian polling systems. Finally, we provide some simulations to surpport the asymptotic results.

\vspace{0.3cm}

\noindent \rule[0.5pt]{14.5cm}{0.6pt}

\section{Introduction}\label{sec1}
In this paper, we analyze a two-queue Markovian polling system.
The random routing discipline is widely used in computer-communication systems with random access schemes and production systems with multi-type products, such as ALOHA and CSMA-CA (Carrier-Sense Multiple-Access Collision-Avoidance) algorithms.
In the modelling of cellular data services, access is randomly assigned to the multitude of users in a cell. More specifically, a time-slot (representing the right for transmission) is assigned to the user with the highest signal-to-noise ratio among all users in a cell, which is so-called opportunistic scheduling.
The opportunistic scheduling aims to improve the bandwidth efficiency by utilizing fading and shadowing of cellular users within a single cell.

Although there is quite an extensive amount of literature on polling systems, only very few papers concern with Markovian polling systems.
The few papers that can be found include \cite{Boxma1989}, where a pseudo-conservation law for mean waiting times was derived.
Later, a more general Markovian polling system with a dependent routing probabilities was studied in \cite{Srinivasan1991}.
More recently, Dorsman considered the applications to wireless random-access networks and analyzed some optimizing problems in \cite{Dorsman2014}.
He also considered the case of two queues specially in \cite{Dorsman2014a}, where the heavy-traffic behavior was investigated with the descendent set approach (DSA) for the first time.

Here we introduce priority policy to differentiate between high priority traffic, like streaming multimedia, and low priority traffic, like web browsing, to improve the QoS (Quality of Service) standard, just like the application of the threshold priority policy in \cite{Liu2014a,Liu2014b}.
Only very few papers treat priorities in polling systems.
Boon first proved the queue lengths at polling epochs do not depend on the service order and hence gave their Probability Generating Functions (PGFs) in priority polling systems in \cite{Boon2010}.
Besides, he applied the notion of delay-cycles, introduced in \cite{Kella1988}, and presented the LSTs of the waiting times.
The model was further extended with multiple priority levels in \cite{Boon2010a} and with mixed gated/exhaustive service discipline in \cite{Boon2009}.
Here, we apply the same methodologies to the Markovian polling systems.

Apart from the aforementioned performances, we also discuss the limiting delays with large switch-over times.
The case of large deterministic switch-over times is explored by using the DSA in combination with the Strong Law of Large Numbers for Renewal Processes, while the delay with large general switch-over times is studied under heavy traffic.

The remainder paper is organized as follows. We first give a detailed description of the model in Section \ref{sec2} and then present the analysis of the non-priority polling system as the preliminary work in Section \ref{sec2a}.
Section \ref{sec3} is dedicated to the derivation of the cycle time, the joint queue lengths at polling epochs and at an arbitrary time, and the waiting times.
The scaled delays in the heavy-traffic and with large switch-over times are explored in Section \ref{sec4} and Section \ref{sec5} respectively.
In Section \ref{sec6}, simulations are undertaken to test the validity of limiting theorems.
We finally conclude the whole procedure and propose some topics for further research in Section \ref{sec7}.

 \section{Model description}\label{sec2}
 We consider a single server Markovian polling model with
 two queues $Q_{1}$ and $Q_{2}$.
 Customers arriving at $Q_{2}$ are referred to type-$2$ customers, while $Q_{1}$ contains two types of customers: type-$H$ and type-$L$ customers. Type-$H$ customers have non-preemptive priority over type-$L$ customers.
 The buffer capacity of each queue is infinite.
 Each type of customers are served in FCFS discipline. Type-$i$ customers arrive independently according to a Poisson process with rate $\lambda_{i}$ and the service times $B_{i}$
 are mutually independent with LST $\widetilde{B}_{i}(s)$ and mean $\mathbb{E}B_i=\frac{1}{\mu_i}$, $i=H,L,2$.
  Then the traffic load of $Q_i$ equals $\rho_i=\lambda_i\mathbb{E}B_i$ ($i=1,2$) with $\lambda_{1}=\lambda_H+\lambda_L$ and $B_{1}=\frac{\lambda_H}{\lambda_{1}}B_{H}+\frac{\lambda_L}{\lambda_{1}}B_{L}$. Let $\rho_i=\frac{\lambda_i}{\mu_i}$ ($i=H,L$). Therefore, $\rho_1=\rho_H+\rho_L$.
It is assumed that each queue is served exhaustively.
 We also assume the total traffic load $\rho=\rho_1+\rho_2<1$, according to the conclusion in \cite{Foss1996}, which implies the system is stable.

 The Markovian routing mechnism means that the next queue polled will be determined by a discrete time Markov chain.
 For convenience, we denote the Markov chain as $M=\{d_n,n\geq 0\}$ with state space $I=\{1,2\}$, where $d_n=i$ means the $n$th polled queue after $t=0$ is $Q_{i}$.
 It is assumed that once the server completes service of $Q_{i}$, it begins a switch-over time $S_{ii}$ for another poll at $Q_{i}$ with probability $p_i<1$ and begins a switch-over time $S_{ij}$ for a poll at $Q_{j}$ with probability $1-p_i$ ($i=1,2$).
 Define
 \begin{eqnarray*}
   \pi_i&=&\lim_{n\rightarrow \infty}Pr\{d_n=i\}, \ \  i\in I, \\
   r_{ij}&=&Pr\{d_n=i | d_{n+1}=j\}, \ \ \  i,j\in I,\  n=0,1,\ldots
 \end{eqnarray*}
With the notations, the average duration of an arbitrary switch-over time is given by $\sigma=\sum_{i=1}^2\sum_{j=1}^2 r_{ij}\pi_jE[S_{ij}]$. And it is easy to compute
\begin{align*}
  &\pi_1=\frac{1-p_2}{2-p_1-p_2}, & &\pi_1=\frac{1-p_1}{2-p_1-p_2}, \\
  &r_{11}=p_1, && r_{12}=1-p_2,\\
  &r_{21}=1-p_1, && r_{22}=p_2.
\end{align*}

\section{Preliminary}\label{sec2a}
To avoid the tedius calculations, we follow the same idea in \cite{Boon2010} to give the joint number of customers at polling epochs.
 This section contains some preliminaries concerning the analysis of non-priority polling system. ``Non-priority'' means that we combine type-H and type-L customers into general type-1 customers with  arriving rate $\lambda_{1}$ and service times $B_{1}$. For compactness, we omit the derivation process, which can be referred to \cite{Dorsman2014a}. It is noted that throughout the paper, for the random variable $X$, $\widetilde{X}(\cdot)$ denotes its PGF or LST.

Let $K_{ij}$ denote the number of type-$j$ customers arriving at $Q_i$ during a busy period. Then its PGF is given by
\begin{equation*}
  \widetilde{K}_{ij}(z)=\widetilde{\theta}_i(\lambda_j(1-z)),\ \ \  i,j=1,2,
\end{equation*}
where $\theta_i$ denotes a busy period in an $M/G/1$ queue with arrival rate $\lambda_i$ and service rate $\mu_i$.

Let $M_{ij}^{(k)}$ be the number of arriving type-$k$ customers during a switch-over time $S_{ij}$. Then the PGF of the joint distribution of $M_{ij}^{(k)}$ is given by
\begin{equation*}
  \widetilde{M}_{ij}(z_1,z_2)=\mathbb{E}[z_1^{M_{ij}^{(1)}}z_2^{M_{ij}^{(2)}}]=\widetilde{S}_{ij}(\lambda_1(1-z_1)+\lambda_2(1-z_2)),\ \ \  i,j=1,2.
\end{equation*}

Along the arguments in \cite{Dorsman2014a}, the PGFs of the joint queue lengths at the polling epochs at $Q_i$, denoted by $\widetilde{V}_{b_i}(z_1,z_2)$, are given by
\begin{align}
 \widetilde{ V}_{b_1}(z_1,z_2)&=r_{11}\widetilde{M}_{11}(z_1,z_2)\prod_{j=0}^\infty a_1(f_1^{(j)}(z_2))+r_{21}\widetilde{M}_{21}(z_1,z_2)\prod_{j=0}^\infty a_2(f_2^{(j)}(z_1)),\label{3.1-1a}\\
  \widetilde{V}_{b_2}(z_1,z_2)&=r_{12}\widetilde{M}_{12}(z_1,z_2)\prod_{j=0}^\infty a_1(f_1^{(j)}(z_2))+r_{22}\widetilde{M}_{22}(z_1,z_2)\prod_{j=0}^\infty a_2(f_2^{(j)}(z_1)),\label{3.1-1b}
  \end{align}
  where
  \begin{align*}
  &f_1(z_2)=\widetilde{K}_{21}(\widetilde{K}_{12}(z_2)), \qquad f_2(z_1)=\widetilde{K}_{12}(\widetilde{K}_{21}(z_1)),\\
  &f_1^{(0)}(z_2)=z_2,\qquad \qquad \qquad f_1^{(j)}(z_2)=f_1(f_1^{(j-1)}(z_2)),\ \ \   j=1,2,\ldots,\\
  &f_2^{(0)}(z_1)=z_1,\qquad\qquad \qquad f_2^{(j)}(z_1)=f_2(f_2^{(j-1)}(z_1)),\ \ \   j=1,2,\ldots,\\
  &a_1(z_2)=\frac{r_{21}\widetilde{M}_{21}(\widetilde{K}_{12}(z_2),z_2)}{1-r_{11}\widetilde{M}_{11}(\widetilde{K}_{12}(z_2),z_2)}\frac{r_{12}\widetilde{M}_{12}(\widetilde{K}_{12}(z_2),f_1(z_2))}{1-r_{22}\widetilde{M}_{22}(\widetilde{K}_{12}(z_2),f_1(z_2))},\\
  &a_2(z_1)=\frac{r_{12}\widetilde{M}_{12}(z_1,\widetilde{K}_{21}(z_1))}{1-r_{22}\widetilde{M}_{22}(z_1,\widetilde{K}_{21}(z_1))}\frac{r_{21}\widetilde{M}_{21}(f_2(z_1),\widetilde{K}_{21}(z_1))}{1-r_{11}\widetilde{M}_{11}(f_2(z_1),\widetilde{K}_{21}(z_1))}.
  \end{align*}
Substituting $z_2=1$ in \eqref{3.1-1a}, we get
\begin{equation}
  \widetilde{V}_{b_1}(z_1,1)=r_{11}\widetilde{M}_{11}(z_1,1)+r_{21}\widetilde{G}_1(z_1),\label{3.1-1c}
  \end{equation}
  where
  \begin{equation}
      \widetilde{G}_1(z) = \widetilde{M}_{21}(z,1)\prod_{j=0}^\infty a_2(f_2^{(j)}(z)),\label{3.1-1e}
    \end{equation}
   Here $\prod_{j=0}^\infty a_2(f_2^{(j)}(z))$ can be interpreted as the total number of currently present type-1 customers whose descendants arrive during all the \emph{inter visit-end periods of $Q_2$} (see \cite{Dorsman2014a}) in the past. Besides, $\widetilde{G}_1(z)$ can be computed by the DSA as follows.
   \begin{equation}\label{3.5-1a}
 \widetilde{G}_1(z)=\widetilde{H}_1(z)\frac{1-r_{11}\widetilde{M}_{11}(z,1)}{r_{21}},
\end{equation}
  where $H_1$ satisfies
    \begin{equation}\label{3.5-1}
    \begin{split}
      \widetilde{H}_1(z)=&\prod_{c=0}^{\infty}\bigg[\widetilde{R}_1\left(\lambda_1(1-\widetilde{A}_{1,c-1}(z))+\lambda_2(1-\widetilde{A}_{2,c}(z))\right) \\
         &\times\widetilde{R}_2\left(\lambda_1(1-\widetilde{A}_{1,c-1}(z))+\lambda_2(1-\widetilde{A}_{2,c-1}(z))\right)\bigg].
    \end{split}
\end{equation}
with
\begin{align*}
\widetilde{R}_1(s)&=\widetilde{S}_{12}(s)\frac{r_{12}}{1-r_{22}\widetilde{S}_{22}(s)}, \\ \widetilde{R}_2(s)&=\widetilde{S}_{21}(s)\frac{r_{21}}{1-r_{11}\widetilde{S}_{11}(s)},\\
  \widetilde{A}_{1,c}(z) &=\widetilde{\theta}_1(\lambda_2(1-\widetilde{A}_{2,c}(z)))=\widetilde{K}_{12}(\widetilde{A}_{2,c}(z)), &\widetilde{A}_{1,-1}(z)&=z, \\
  \widetilde{A}_{2,c}(z) &=\widetilde{\theta}_2(\lambda_1(1-\widetilde{A}_{1,c-1}(z)))=\widetilde{K}_{21}(\widetilde{A}_{1,c-1}(z)), &\widetilde{A}_{2,-1}(z)&=1.
\end{align*}

Let $V_i$ denote the visit time at $Q_i$, $I_i$ denote the intervisit time of $Q_i$ and $C_i$ denote the cycle time starting with a visit beginning at $Q_i$ (i=1,2).
For exhaustive service discipline, it is evident that
  \begin{align}
  \widetilde{I}_1(s)&=\widetilde{V}_{b_1}(1-\frac{s}{\lambda_1},1), &\widetilde{C}_1(s)&=\widetilde{V}_{b_1}(\widetilde{\theta}_1(s)-\frac{s}{\lambda_1},1),\label{2.1}\\
  \widetilde{I}_2(s)&=\widetilde{V}_{b_2}(1,1-\frac{s}{\lambda_2}), &\widetilde{C}_2(s)&=\widetilde{V}_{b_2}(1,\widetilde{\theta}_2(s)-\frac{s}{\lambda_2}).\label{2.2}
\end{align}
In addition, with the balance arguments, the first moment can be obtained more easily. Here we just mention the results that $\mathbb{E}C_i=\frac{\sigma}{\pi_i(1-\rho)}$ and $\mathbb{E}I_i=\frac{\sigma(1-\rho_i)}{\pi_i(1-\rho)}$ ($i=1,2$). See \cite{Dorsman2014} for more details.

\section{Analysis}\label{sec3}
For a random variable $X$, we denote its PGF in priority polling system by $\widetilde{X}(\cdot,\cdot,\cdot)$ and let $\widetilde{X}(\cdot,\cdot)$ be the analogous quantities for non-priority system.

\subsection{Joint number of customers at polling epochs and cycle time}\label{subsec3.1}
Since the joint number of customers at polling epochs and the cycle time do not depend on the order of service, they remain the same as that in the non-priority system, which is rigorously proved in Lemma 1 in \cite{Boon2010}. Hence, we have the following theorems.
\begin{thm}\label{thm1}
The PGFs of the joint number of customers at each polling epochs are given by
  \begin{equation*}
  \widetilde{V}_{b_i}(z_H,z_L,z_2)=\widetilde{V}_{b_i}(\frac{\lambda_Hz_H+\lambda_Lz_L}{\lambda_1},z_2),\ \  \   i=1,2.\label{3.1-1}
  \end{equation*}
\end{thm}
In addition, \eqref{2.1} and \eqref{2.2} are still valid for the LSTs of the cycle times and intervisit times.

\subsection{Waiting time}\label{subsec3.3}
From the viewpoint of a type-$H$ customer, the polling system is an $M/G/1$ queue with multiple vacations: the intervisit time $I_1$ and the service time of a type-$L$ customer $B_L$.
Hence, we can use the concept of delay-cycles, introduced in \cite{Kella1988}, a methodology which is devoted to deriving the waiting times in $M/G/1$ queue with priorities and vacations, to compute the waiting time of a type-$H$ customer.

   The delay-cycle for a tagged type-$H$ customer is a cycle that starts with a certain initial delay that no type-$H$ customer is waiting in line and terminates at the moment that no type-$H$ customers are present in the system.
   For simplicity, the delay-cycle with initial delay $I_1$ is denoted by $I_H$ cycle and the delay-cycle with initial delay $B_L$ is denoted by $L_H$ cycle.
   It is noted that the arrival of a tagged type-$H$ customer must take place within some delay-cycle.
   The fraction of the time that the system is in an $L_H$ cycle equals $\lambda_L\cdot\frac{\mathbb{E}B_L}{1-\rho_H}=\frac{\rho_L}{1-\rho_H}$ and the fraction in an $I_H$ cycle equals $1-\frac{\rho_L}{1-\rho_H}=\frac{1-\rho_1}{1-\rho_H}$.
   By the Fuhrmann-Cooper decomposition Theorem in \cite{fuhrmann1985}, the LST of the waiting time of a type-$H$ customer is expressed by
   \begin{align*}\label{3.3-9}
     \widetilde{W}_H(s)&=\frac{(1-\rho_H)s}{s-\lambda_H(1-\widetilde{B}_H(s))}\left[\frac{1-\rho_1}{1-\rho_H}\frac{1-\widetilde{I}_1(s)}{s\mathbb{E}I_1}+\frac{\rho_L}{1-\rho_H}\frac{1-\widetilde{B}_L(s)}{s\mathbb{E}B_L}\right]\\
     &=\frac{\pi_1(1-\rho)\left[1-\widetilde{V}_{b_1}(1-\frac{s}{\lambda_1},1)\right]}{\sigma[s-\lambda_H(1-\widetilde{B}_H(s))]}+\frac{\rho_L\left[1-\widetilde{B}_L(s)\right]}{[s-\lambda_H(1-\widetilde{B}_H(s))]\mathbb{E}B_L}.
   \end{align*}

   Now we introduce a \emph{completion time} with LST of
 \begin{equation}\label{3.3-7a}
   \widetilde{B}_{L'}(s)=\widetilde{B}_L(s+\lambda_H(1-\widetilde{\theta}_H(s))),
 \end{equation}
 where $\theta_H$ denotes a busy period in an $M/G/1$ queue with arrival rate $\lambda_H$ and service rate $\mu_H$, which is actually the service time of a type-$L$ customer, plus the service time of its type-$H$ descendants.
 Then for a type-$L$ customer, the polling system is an $M/G/1$ queue with vacation $I_{1'}$ of LST:
   \begin{equation*}
  \widetilde{I}_{1'}(s)=\widetilde{I}_1(s+\lambda_H(1-\widetilde{\theta}_H(s))).
  \end{equation*}
  Customers arrive with rate $\lambda_L$ and have service time $B_{L'}$. The traffic load equals $\rho_{L'}=\frac{\rho_L}{1-\rho_H}$. By the Fuhrmann-Cooper decomposition Theorem in \cite{fuhrmann1985}, it is readily to obtain
 \begin{align*}
    \widetilde{W}_L(s)&=\frac{(1-\rho_{L'})s}{s-\lambda_L(1-\widetilde{B}_{L'}(s))}\frac{1-\widetilde{I}_{1'}(s)}{s\mathbb{E}I_{1'}}\\
    &=\frac{\pi_1(1-\rho)}{\sigma[s-\lambda_L(1-\widetilde{B}_{L'}(s))]}\left[1-\widetilde{V}_{b_1}(\widetilde{\theta}_H(s),1-\frac{s}{\lambda_L},1)\right].
   \end{align*}

   Note that the waiting time of a type-$2$ customer remains the same as that in the non-priority system. Hence, applying the conclusion in \cite{Borst1997} leads to
   \begin{equation*}\label{3.3-8}
    \widetilde{W}_2(s)=\frac{\pi_2(1-\rho)}{\sigma[s-\lambda_2(1-\widetilde{B}_2(s))]}\left[1-\widetilde{V}_{b_2}(1,1,1-\frac{s}{\lambda_2})\right].
   \end{equation*}

   Now all the derivations in this subsection can be summarized as:
   \begin{thm}\label{thm3}
   The LST of the waiting times of each type customer are expressed by
     \begin{eqnarray}
     \widetilde{W}_H(s) &=& \frac{\pi_1(1-\rho)\left[1-\widetilde{V}_{b_1}(1-\frac{s}{\lambda_1},1)\right]}{\sigma[s-\lambda_H(1-\widetilde{B}_H(s))]}+\frac{\lambda_L\left[1-\widetilde{B}_L(s)\right]}{[s-\lambda_H(1-\widetilde{B}_H(s))]},\label{3.3-7}\\
      \widetilde{W}_L(s)&=&\frac{\pi_1(1-\rho)}{\sigma[s-\lambda_L(1-\widetilde{B}_{L'}(s))]}\left[1-\widetilde{V}_{b_1}(\widetilde{\theta}_H(s),1-\frac{s}{\lambda_L},1)\right],\label{3.3-8} \\
       \widetilde{W}_2(s)&=&\frac{\pi_2(1-\rho)}{\sigma[s-\lambda_2(1-\widetilde{B}_2(s))]}\left[1-\widetilde{V}_{b_2}(1,1,1-\frac{s}{\lambda_2})\right].\label{3.3-9}
     \end{eqnarray}
   \end{thm}
\subsection{Joint number of customers at an arbitrary time}\label{subsec3.4}
Unlike the joint number of customers at polling epochs, the joint number of customers at an arbitrary time does depend on the service order.
Now we turn this priority model into an equivalent non-priority model by dividing $Q_1$ into two queues $Q_H$ and $Q_L$.
To distinguish from the definitions of type-$H$ and type-$L$ customers, we take customers at $Q_H$ as type-$H'$ customers and customers at $Q_L$ as type-$L'$ customers. Type-$L'$ customers arrive with intensity $\lambda_L$, while type-$H'$ customers belong to \emph{smart customers} described in \cite{Boon2010b}, that if the server is not at $Q_{L}$, type-$H'$ customers arrive with intensity $\lambda_H$, otherwise there are no type-$H'$ customers arriving. In addition, type-$H'$ customers have service requirement $B_H$, and type-$L'$ customers have service requirement $B_{L'}$ with LST $\widetilde{B}_{L'}(s)$ defined in \eqref{3.3-7a}. Hence the traffic load of $Q_{L}$ equals $\rho_{L'}=\frac{\rho_L}{1-\rho_H}$ and the traffic load of $Q_H$ equals $\rho_{H'}=\rho_1-\rho_{L'}=\frac{(1-\rho_1)\rho_H}{1-\rho_H}$.

 Obviously, the non-priority model has the same joint queue lengths at visit beginning and visit ending of $Q_2$ as the original model. Besides, $\widetilde{V}_{b_i}(\cdot,\cdot,\cdot)$ and $\widetilde{V}_{c_i}(\cdot,\cdot,\cdot)$, the PGF of the joint queue lengths at visit beginning and visit ending of $Q_i$ ($i=H,L$) can be expressed by:
  \begin{align}
  \widetilde{V}_{b_H}(z_H,z_L,z_2)&=\widetilde{V}_{b_1}(z_H,z_L,z_2)=\widetilde{V}_{b_1}(\frac{\lambda_Hz_H+\lambda_Lz_L}{\lambda_1},z_2),\label{3.3-1}\\
  \widetilde{V}_{b_L}(z_H,z_L,z_2)&=\widetilde{V}_{c_H}(z_H,z_L,z_2)=\widetilde{V}_{b_1}(h_H(z_L,z_2),z_L,z_2),\label{3.3-2}
  \end{align}
  where $h_H(z_L,z_2)=\widetilde{\theta}_H(\lambda_L(1-z_L)+\lambda_2(1-z_2))$.

Let $L_i$ denote the number of type-$i$ customers at an arbitrary time ($i=H,L,2$) in the original model. For writing convenience, we take the notations $\mathbf{z}$ short for $(z_H,z_L,z_2)$ and $\mathbf{\lambda}(\mathbf{z})=\lambda_H(1-z_H)+\lambda_L(1-z_L)+\lambda_2(1-z_2)$.
By conditioning on the location $P$ of the server ($P\in \{V_i:i=H,L,2\}\cup \{S_{ij}:i,j=1,2\}$) in the equivalent model, we circumvent the PGFs of the joint number of customers at an arbitrary time, denoted by $\widetilde{L}(\mathbf{z})$. Let $L_i^{(V_j)}$ and $L_i^{(S_{jk})}$ denote the number of type-$i$ customers at an arbitrary time during $V_j$ ($j=H,L,2$) and $S_{jk}$ ($j,k=1,2$) respectively. Then
\begin{equation}\label{3.4-1}
 \widetilde{L}(\mathbf{z})=\sum_{i=H,L}\rho_{i'}\widetilde{L}^{(V_i)}(\mathbf{z})+\rho_{2}\widetilde{L}^{(V_2)}(\mathbf{z})+\frac{1-\rho}{\sigma}\sum_{i=1}^2\sum_{j=1}^2r_{ij}\pi_j\mathbb{E}S_{ij}\widetilde{L}^{(S_{ij})}(\mathbf{z}).
\end{equation}
Let $\widetilde{S}_{b}^{(V_i)}(z)$ denote the joint number of customers at a service beginning of $Q_i$ and $X^{past}$ denote the elapsed life of $X$. Then
\begin{equation}\label{3.4-2}
  \begin{split}
    \widetilde{L}^{(V_H)}(\mathbf{z})&= \widetilde{S_{b}}^{(V_H)}(\mathbf{z}) \widetilde{B_H}^{past}(\mathbf{\lambda}(\mathbf{z}))\\
      & =\frac{\gamma_{H'}z_H\left[\widetilde{V}_{b_H}(\mathbf{z})-\widetilde{V}_{c_H}(\mathbf{z})\right]}{z_H-\widetilde{B}_H(\mathbf{\lambda}(\mathbf{z}))}\frac{1-\widetilde{B}_H(\mathbf{\lambda}(\mathbf{z}))}{\mathbf{\lambda}(\mathbf{z})\mathbb{E}B_H}\\
  &=\frac{\pi_1(1-\rho)}{\rho_{H'}\sigma}\frac{z_H\left[\widetilde{V}_{b_H}(\mathbf{z})-\widetilde{V}_{b_H}(h_H(\mathbf{z}),z_L,z_2)\right]}{z_H-\widetilde{B}_H(\mathbf{\lambda}(\mathbf{z}))}\frac{1-\widetilde{B}_H(\mathbf{\lambda}(\mathbf{z}))}{\mathbf{\lambda}(\mathbf{z})},
  \end{split}
\end{equation}
where the last equation follows from $\gamma_{H'}=\frac{1-\rho}{1-\rho_{L'}}\frac{\pi_1}{\lambda_H\sigma}$ since there are no arrivals of type-$H$ customers during $V_L$.

Along the same argument, we get
\begin{equation}\label{3.4-3}
    \widetilde{L}^{(V_2)}(\mathbf{z})= \frac{\pi_2(1-\rho)}{\rho_2\sigma}\frac{z_2\left[\widetilde{V}_{b_2}(\mathbf{z})-\widetilde{V}_{b_2}(h_2(z_H,z_L,\mathbf{z}))\right]}{z_2-\widetilde{B}_2(\mathbf{\lambda}(\mathbf{z}))}\frac{1-\widetilde{B}_2(\mathbf{\lambda}(\mathbf{z}))}{\mathbf{\lambda}(\mathbf{z})}.
\end{equation}

Let $\widetilde{L}^{(V_{L'})}(\mathbf{z})$ denote the PGF of the joint number of customers at an arbitrary time during $B_{L'}$. Then for $\widetilde{L}^{(V_L)}(\mathbf{z})$, using the same method yields
\begin{equation}\label{3.4-4}
\begin{split}
    \widetilde{L}^{(V_L)}(\mathbf{z})&= \widetilde{S_{b}}^{(V_L)}(\mathbf{z}) \widetilde{L}^{(V_{L'})}(\mathbf{z})\\
    &= \frac{\pi_1(1-\rho)}{\lambda_L\sigma}\frac{z_L\left[\widetilde{V}_{b_L}(\mathbf{z})-\widetilde{V}_{b_L}(1,h_{L'}(z_2),z_2)\right]}{z_L-\widetilde{B}_{L'}(\mathbf{\lambda}(\mathbf{z}))}\widetilde{L}^{(V_{L'})}(\mathbf{z}),
\end{split}
\end{equation}
where $h_{L'}(z_2)=\widetilde{\theta}_{L'}(\lambda_2(1-z_2))$ and $\theta_{L'}$ is a busy time of an $M/G/1$ queue with arriving rate $\lambda_L$ and service requirement $B_{L'}$.

However, $\widetilde{L}^{(V_{L'})}(\mathbf{z})\neq \widetilde{B_{L'}}^{past}(\mathbf{\lambda}(\mathbf{z}))$, since the tagged time during a service time $B_{L'}$ depends on the number of arriving type-$H$ descendants of the type-$L$ customer. To solve $\widetilde{L}^{(V_{L'})}(\mathbf{z})$, we introduce a cyclic polling system with three queues: $Q_{H^*}, Q_{L^*}, Q_{2^*}$ (in cyclic order). Customers arrive at $Q_{i^*}$ according to a Poisson Process with intensity $\lambda_i$. We refer the customers at $Q_{i^*}$ as type-$i^*$ customers. Type-${H^*}$ customer has service requirement $B_H$ and other customers has infinitely small service requirement, which means that once their service begins, it can be completed immediately. It is assumed that there are no switch-over times while  $Q_{H^*}$ has a setup time $B_L$. It is easy to see that $\widetilde{L}^{(V_{L'})}(\mathbf{z})$ is just the PGF of the joint queue length at an arbitrary time of this system. Hence, we have
\begin{equation}\label{3.4-5}
\begin{split}
   \widetilde{L}^{(V_{L'})}(\mathbf{z})=&\rho_H\frac{1-\rho_H}{\rho_H}\frac{z_H\left[\widetilde{B_L}(\mathbf{\lambda}(\mathbf{z}))-\widetilde{B_L}(\mathbf{\lambda}(h_H(\mathbf{z}),z_L,z_2))\right]}{z_H-\widetilde{B}_H(\mathbf{\lambda}(\mathbf{z}))}\frac{1-\widetilde{B}_H(\mathbf{\lambda}(\mathbf{z}))}{\mathbf{\lambda}(\mathbf{z})\mathbb{E}B_H}\\
 &+(1-\rho_H)\frac{1-\widetilde{B}_L(\mathbf{\lambda}(\mathbf{z}))}{\mathbf{\lambda}(\mathbf{z})\mathbb{E}B_L}
\end{split}
\end{equation}
Substituting \eqref{3.4-5} into \eqref{3.4-4} yields the expression of $\widetilde{L}^{(V_L)}(\mathbf{z})$.
Besides, $L_i^{(S_{jk})}$ can be expressed as
\begin{eqnarray}
  \widetilde{L}^{(S_{1j})}(\mathbf{z})&=&\widetilde{V}_{b_1}(\widetilde{K}_{12}(z_2),z_2)\frac{1-M_{1j}(\mathbf{z})}{\mathbf{\lambda}(\mathbf{z})\mathbb{E}S_{1j}}, \label{3.4-7}\\
  \widetilde{L}^{(S_{2j})}(\mathbf{z})&=&\widetilde{V}_{b_2}\left(\frac{\lambda_Hz_H+\lambda_Lz_L}{\lambda_1},\widetilde{K}_{21}(\frac{\lambda_Hz_H+\lambda_Lz_L}{\lambda_1})\right)\frac{1-M_{2j}(\mathbf{z})}{\mathbf{\lambda}(\mathbf{z})\mathbb{E}S_{2j}}. \label{3.4-8}
\end{eqnarray}
Then the PGF of the joint number of customers at an arbitrary time can be computed by combining \eqref{3.4-2}-\eqref{3.4-8} with \eqref{3.4-1}.
\section{Heavy traffic asymptotics}\label{sec4}
In this section, we increase $\Lambda=\lambda_1+\lambda_2+\lambda_3$ to make the traffic load $\rho\rightarrow 1$ while keeping the ratios $\lambda_H:\lambda_L:\lambda_2$ fixed.
When $\rho\rightarrow 1$, the queue lengths tend to infinity.
Hence, we are devoted to studying the distributions of the scaled delay $\mathcal{W}_i=(1-\rho)W_i, i=H,L,2$.
For the Markovian priority polling system, the DSA can be the first choice.
Van der Mei \cite{VanderMei1999,VanderMei1999a} introduced this technique to give all moments and the distributions of the delay in the heavy traffic.
Dorsman also applied this technique to a non-priority Markovian polling system in \cite{Dorsman2014a}.
Throughout the remainder paper, we introduce a notation $\hat{x}$ such that $\hat{x}=\frac{x}{\rho}$, for example, $\hat{\rho}_H=\frac{\rho_H}{\rho}$.

 For preparations, we introduce two Lemmas first. The first one is the Theorem of Method of Moments introduced in \cite{VanderMei1999a}, which is applied to prove a gamma distributed limiting random variable.
 \begin{lem}\label{lem1}
 Let $\Gamma$ be a gamma-distributed random variable with shape parameter $\alpha+1$ and scale parameter $\mu$. Let $\{Y_n\}$ be a sequence of random variables with finite moments, satisfying
 \begin{equation*}
   \lim_{n\rightarrow \infty}\mathbb{E}Y_n^{k}=\mathbb{E}\Gamma^{k}, \ \ \ k=1,2,\ldots
 \end{equation*}
 Then $Y_n\overset{d}{\rightarrow} \Gamma$.
 \end{lem}

Lemma \ref{lem2} is devoted to the expressions of the $k$th derivative $[f(g)]^{(k)}(x)$ of a general composite function $f(g(x))$, which is Theorem 1 presented in \cite{Mei2000}.
 \begin{lem}\label{lem2}
 For $k=1,2,\ldots,$ the $k$th derivative of composite function $f(g(x))$ is given by
 \begin{equation*}
   \left[f(g)\right]^{(k)}(x)=\sum_{\mathbf{m}^{(k)}\in S_k}c_k\left(\mathbf{m}^{(k)}\right)f^{(l_k)}\left(g(x)\right)\prod_{i=1}^k\left(g^{(i)}(x)\right)^{m_i}, l_k=\sum_{j=1}^km_j,
 \end{equation*}
 where
 \begin{equation*}
   S_k:=\left\{\mathbf{m}^{(k)}=(m_1,\ldots,m_k): m_j \text{are non-negative integers with} \sum_{j=1}^kjm_j=k\right\},
 \end{equation*}
 and $c_k\left(\mathbf{m}^{(k)}\right)$ can be calculated in the following recursive way: $c_1(1)=1$ and
 \begin{equation*}
   \begin{split}
     c_k\left(\mathbf{m}^{(k)}\right) =&c_{k-1}(m_1-1,m_2,\ldots,m_{k-1})I_{\{m_1>0\}}  \\
       & +\sum_{j=1}^{k-1}(m_j+1)c_{k-1}(m_1,\ldots,m_{j-1},m_{j}+1,m_{j+1}-1,m_{j+2},\ldots,m_{k-1})\\
       &\times I_{\{m_{j+1}>0\}},\ \  \ k=2,3,\ldots.
   \end{split}
 \end{equation*}
 \end{lem}
Applying Lemma \ref{lem1} and Lemma \ref{lem2} to the DSA expressions \eqref{3.5-1a} of $G_1$, we obtain the following Proposition, whose proof can be found in \cite{Dorsman2014a}.
 \begin{pro}\label{pro1}
 As $\rho\rightarrow 1$, we have
 \begin{equation*}
   \lim_{\rho\rightarrow 1}\mathbb{E}[(1-\rho)^kG_1^{k}]=\frac{\hat{\lambda}_1^k\prod_{j=0}^{k-1}(\alpha+j)}{\nu_1^k}, \ \ \ k=1,2,\ldots,
 \end{equation*}
where
 \begin{align*}
       &\alpha= \frac{2\hat{\rho}_1\hat{\rho}_2\mathbb{E}S^{tot}}{\hat{\lambda}_H\mathbb{E}B_H^2+\hat{\lambda}_L\mathbb{E}B_L^2+\hat{\lambda}_2\mathbb{E}B_2^2},\\
        &\nu_1= \frac{2\hat{\rho}_1}{\hat{\lambda}_H\mathbb{E}B_H^2+\hat{\lambda}_L\mathbb{E}B_L^2+\hat{\lambda_2}\mathbb{E}B_2^2},\\
      & \mathbb{E}S^{tot}=\frac{p_1}{1-p_1}\mathbb{E}S_{11}+\mathbb{E}S_{12}+\mathbb{E}S_{21}+\frac{p_2}{1-p_2}\mathbb{E}S_{22}=\frac{2-p_1-p_2}{(1-p_1)(1-p_2)}\sigma.
     \end{align*}
 Further, the random variable $(1-\rho)G_1$ converges in distribution to a Gamma-distributed random variable with shape parameter $\alpha+1$ and scale parameter $\frac{\nu_1}{\hat{\lambda}_1}$ respectively.
 \end{pro}

 Based on this Proposition, we can further get the following equations:

 \begin{pro}\label{pro2}
 As $\rho\rightarrow 1$, we have the limiting equations:
 \begin{align}
   &\lim_{\rho\rightarrow 1}\widetilde{G}_1(1-\frac{(1-\rho)s}{\lambda_1})=\left(\frac{\nu_1}{\mu_1+s}\right)^\alpha,\label{3.4-9} \\
   &\lim_{\rho\rightarrow 1}\widetilde{G}_1\Big(\frac{\lambda_H}{\lambda_1}\widetilde{\theta}_H((1-\rho)s)+\frac{\lambda_L}{\lambda_1}(1-\frac{(1-\rho)s}{\lambda_L})\Big)=\left(\frac{(1-\hat{\rho}_H)\nu_1}{(1-\hat{\rho}_H)\nu_1+s}\right)^\alpha.\label{3.4-10}
 \end{align}
 \end{pro}

 \begin{proof}
 The proof of the first equation can be referred to \cite{Dorsman2014a}. For brevity, we only give the proof of the second equation.

   Note that there exists a single random variable $Y$ that satisfies
   \begin{equation*}
     \widetilde{Y}(s)=\mathbb{E}e^{-sY}=\widetilde{G}_1\Big(\frac{\lambda_H}{\lambda_1}\widetilde{\theta}_H((1-\rho)s)+\frac{\lambda_L}{\lambda_1}(1-\frac{(1-\rho)s}{\lambda_L})\Big)=f(g(s)),
   \end{equation*}
   where $f(z)=\widetilde{G}_1(z)$ and $g(s)=\frac{\lambda_H}{\lambda_1}\widetilde{\theta}_H(s)+\frac{\lambda_L}{\lambda_1}(1-\frac{s}{\lambda_L})$. It is easy to obtain
   \begin{align*}
     f^{(k)}(1)&=\mathbb{E}G_1^{k},\ \ \  k=1,2,\ldots, \\
     g^{(1)}(0)&=-\frac{\lambda_H}{\lambda_1}(1-\rho)\mathbb{E}\theta_H-\frac{1-\rho}{\lambda_1}=-\frac{1-\rho}{\lambda_1(1-\rho_H)},\\
     g^{(k)}(0)&=(-1)^k\frac{\lambda_H}{\lambda_1}(1-\rho)^k\mathbb{E}\theta_H^{k},\ \ \  k=2,3,\ldots.
   \end{align*}
 Then
 \begin{align*}
   \mathbb{E}Y^{k}=&(-1)^k \left[f(g)\right]^{(k)}(0)\\
   =&\sum_{\mathbf{m}^{(k)}\in S_k}c_k\left(\mathbf{m}^{(k)}\right)\mathbb{E}G_1^{l_k}\left(\frac{1-\rho}{\lambda_1(1-\rho_H)}\right)^{m_1}\prod_{i=2}^k\left(\frac{\lambda_H}{\lambda_1}(1-\rho)^i\mathbb{E}\theta_H^{i}\right)^{m_i}\\
   =&\sum_{\mathbf{m}^{(k)}\in S_k}(1-\rho)^kc_k\left(\mathbf{m}^{(k)}\right)\mathbb{E}G_1^{l_k}\left(\frac{1}{\lambda_1(1-\rho_H)}\right)^{m_1}\prod_{i=2}^k\left(\frac{\lambda_H}{\lambda_1}\mathbb{E}\theta_H^{i}\right)^{m_i}\\
   &\prod_{i=2}^k\left(\frac{\lambda_H}{\lambda_1}\mathbb{E}\theta_H^{i}\right)^{m_i}.
 \end{align*}
 Since $l_k=k$ holds iff $\mathbf{m}^{(k)}=(k,0,\ldots,0)$. Hence,
 \begin{equation*}
    \lim_{\rho\rightarrow 1}\mathbb{E}Y^{k}=\lim_{\rho\rightarrow 1}\left[(1-\rho)^k\mathbb{E}G_1^{k}\right]\left(\frac{1}{\lambda_1(1-\rho_H)}\right)^{k} =\frac{\prod_{j=0}^{k-1}(\alpha+j)}{[\nu_1(1-\hat{\rho}_H)]^k},
 \end{equation*}
 where the second equation follows from Proposition \ref{pro1}. By Lemma \ref{lem1}, equation \eqref{3.4-10} holds, which concludes the proof.
  \end{proof}

Now we present the main theorem to give the limiting distributions of the scaled delay.
 \begin{thm}\label{thm4}
 For $0\leq p_1,p_2<1$, the LSTs of the limiting scaled waiting time distributions are given by
 \begin{align}
     \lim_{\rho\rightarrow 1}\widetilde{\mathcal{W}}_H(s)&=(1-\hat{\rho}_{L'})\frac{1}{s(1-\hat{\rho}_1)\mathbb{E}S^{tot}}\left[1-\left(\frac{\nu_1}{\nu_1+s}\right)^\alpha\right]+\hat{\rho}_{L'},\label{3.4-15}\\
     \lim_{\rho\rightarrow 1}\widetilde{\mathcal{W}}_L(s)&=\frac{1-\hat{\rho}_H}{s(1-\hat{\rho}_1)\mathbb{E}S^{tot}}\left[1-\left(\frac{(1-\hat{\rho}_H)\nu_1}{(1-\hat{\rho}_H)\nu_1+s}\right)^\alpha\right],\label{3.4-16}\\
     \lim_{\rho\rightarrow 1}\widetilde{\mathcal{W}}_2(s)&=\frac{1}{s(1-\hat{\rho}_2)\mathbb{E}S^{tot}}\left[1-\left(\frac{\nu_2}{\nu_2+s}\right)^\alpha\right],\label{3.4-17}
     \end{align}
     where $\alpha$, $\nu_1$ and $\mathbb{E}S^{tot}$ are defined in Proposition \ref{pro1} and $\nu_2$ is defined in the same way as $\nu_1$ by $\nu_2= \frac{2\hat{\rho}_2}{\hat{\lambda}_H\mathbb{E}B_H^2+\hat{\lambda}_L\mathbb{E}B_L^2+\hat{\lambda_2}\mathbb{E}B_2^2}$.
     Equivalently,
    \begin{align*}
     \lim_{\rho\rightarrow 1}Pr(\mathcal{W}_H\leq t)&=(1-\hat{\rho}_{L'})Pr(UI_H\leq t)+\hat{\rho}_{L'},\\
     \lim_{\rho\rightarrow 1}Pr(\mathcal{W}_L\leq t)&=Pr(UI_L\leq t),\\
     \lim_{\rho\rightarrow 1}Pr(\mathcal{W}_2\leq t)&=Pr(UI_2\leq t),
     \end{align*}
     where $U$ is a uniformly $[0,1]$ distributed random variable and $I_i$ is a Gamma distributed random variable with shape parameter $\alpha+1$ and scaled parameter $\omega_{i}$ ($i=H,L,2$) with
     \begin{align*}
       \omega_H&=\nu_1, & \omega_L&=(1-\hat{\rho}_H)\nu_1, & \omega_2&=\nu_2.
     \end{align*}
Besides, $U$ and $I_i$ are mutually independent ($i=H,L,2$).
 \end{thm}

 \begin{proof}
 Rewriting equation \eqref{3.3-7} by \eqref{3.1-1c} leads to an alternative expression
  \begin{align}
     \widetilde{W}_H(s) =& \frac{\pi_1(1-\rho)\left[1-r_{11}\widetilde{M}_{11}(1-\frac{s}{\lambda_1},1)-r_{21}\widetilde{G}_1(1-\frac{s}{\lambda_1})\right]}{\sigma[s-\lambda_H(1-\widetilde{B}_H(s))]}\nonumber\\
     &+\frac{\rho_L\left[1-\widetilde{B}_L(s)\right]}{[s-\lambda_H(1-\widetilde{B}_H(s))]\mathbb{E}B_L},\label{3.4-12}
     \end{align}
 Taking the limit of \eqref{3.4-12} leads to
 \begin{align*}
     \lim_{\rho\rightarrow 1}\widetilde{\mathcal{W}}_H(s) =&\lim_{\rho\rightarrow 1}\widetilde{W}_H((1-\rho)s) \\
       = &\lim_{\rho\rightarrow 1} \frac{\pi_1(1-\rho)}{\sigma[(1-\rho)s-\lambda_H(1-\widetilde{B}_H((1-\rho)s))]}\\
       &\times \lim_{\rho\rightarrow 1}\left[1-r_{11}\widetilde{M}_{11}(1-\frac{(1-\rho)s}{\lambda_1},1)-r_{21}\widetilde{G}_1(1-\frac{(1-\rho)s}{\lambda_1})\right]\\
       &+\lim_{\rho\rightarrow 1}\frac{(1-\rho)\rho_L}{[(1-\rho)s-\lambda_H(1-\widetilde{B}_H((1-\rho)s))]\mathbb{E}B_L}\lim_{\rho\rightarrow 1}\frac{1-\widetilde{B}_L((1-\rho)s)}{\mathbb{E}B_L(1-\rho)}\\
       =&\frac{\pi_1}{\sigma s(1-\hat{\rho}_H)}\left[1-r_{11}-r_{21}\left(\frac{\nu_1}{\nu_1+s}\right)^\alpha\right]+\frac{\rho_L}{1-\rho_H},
 \end{align*}
 where the last equation follows from \eqref{3.4-9}. Since $\frac{\pi_1}{\sigma}=\frac{1}{r_{21}\mathbb{E}S^{tot}}$, the above equation is equivalent to \eqref{3.4-15}.

  Rewriting equation \eqref{3.3-8} by \eqref{3.1-1c} yields
  \begin{align}
   \widetilde{W}_L(s)=&\frac{\pi_1(1-\rho)}{\sigma[s-\lambda_L(1-\widetilde{B}'_L(s))]}\bigg[1-r_{11}\widetilde{M}_{11}\Big(\frac{\lambda_H}{\lambda_1}\widetilde{\theta}_H(s)+\frac{\lambda_L}{\lambda_1}(1-\frac{s}{\lambda_L}),1\Big)\nonumber\\
      &-r_{21}\widetilde{G}_1\Big(\frac{\lambda_H}{\lambda_1}\widetilde{\theta}_H(s)+\frac{\lambda_L}{\lambda_1}(1-\frac{s}{\lambda_L})\Big)\bigg]. \label{3.4-13}
  \end{align}
  Using Proposition \ref{pro2} to \eqref{3.4-13}, equations \eqref{3.4-16} can be proved in the similar way as well as \eqref{3.4-17}. The proofs are omitted here.
 \end{proof}

\section{Asymptotics with large switch-over times}\label{sec5}
In this section, we consider the limiting distributions of the scaled delay when the average total switch-over time $\sigma\rightarrow \infty$, which is equivalent to $\mathbb{E}S^{tot}\rightarrow \infty$. For convenience, we take $r=\mathbb{E}S^{tot}$ for short.
\subsection{Asymptotics for deterministic switch-over times}\label{subsec5.1}
In this subsection, we consider the case of the deterministic switch-over times. In this case, \eqref{3.5-1} becomes
\begin{equation}\label{3.5-1b}
  \widetilde{H}_1(z)=\mathrm{e}^{-r_1\sum\limits_{c=0}^{\infty}\left[\lambda_1(1-\widetilde{A}_{1,c-1}(z))+\lambda_2(1-\widetilde{A}_{2,c}(z))\right]-r_2\sum\limits_{c=0}^{\infty}\left[\lambda_1(1-\widetilde{A}_{1,c-1}(z))+\lambda_2(1-\widetilde{A}_{2,c-1}(z))\right]},
\end{equation}
where $r_1=\mathbb{E}R_1=\mathbb{E}S_{12}+\frac{p_2}{1-p_2}\mathbb{E}S_{22}$ and $r_2=\mathbb{E}R_2=\frac{p_1}{1-p_1}\mathbb{E}S_{11}+\mathbb{E}S_{21}$. Let $\alpha_{i,c}^{(k)}$ denote the $k$th moment of $\widetilde{A}_{i,c}(z)$ ($k=1,2,\ldots$). By differentiating $\widetilde{A}_{i,c}(z)$, we can compute $\alpha_{i,c}^{(1)}$ in the recursive way:
\begin{align}\label{3.5-2}
  \alpha_{1,c}^{(1)} & =\mathbb{E}\theta_1\lambda_2\alpha_{2,c}^{(1)}, & \alpha_{2,c}^{(1)} & =\mathbb{E}\theta_2\lambda_1\alpha_{1,c-1}^{(1)}.
\end{align}
Rewriting equation \eqref{3.5-2} to a matrix form yields
\begin{equation}\label{3.5-3}
  \left(\begin{matrix}
  \alpha_{1,c}^{(1)} \\
   \alpha_{2,c}^{(1)}
   \end{matrix}\right)
   =\left(\begin{matrix}
  \lambda_1\lambda_2\mathbb{E}\theta_1\mathbb{E}\theta_2 &0 \\
  \lambda_1\mathbb{E}\theta_2& 0
   \end{matrix}\right)
   \left(\begin{matrix}
  \alpha_{1,c-1}^{(1)} \\
   \alpha_{2,c-1}^{(1)}
   \end{matrix}\right)
   =\left(\begin{matrix}
  \lambda_1\lambda_2\mathbb{E}\theta_1\mathbb{E}\theta_2 &0 \\
  \lambda_1\mathbb{E}\theta_2& 0
   \end{matrix}\right)^{c+1}
   \left(\begin{matrix}
  1 \\
   0
   \end{matrix}\right).
\end{equation}
By differentiating equation \eqref{3.5-1b}, after substituting \eqref{3.5-3}, we obtain
\begin{align}\label{3.5-3a}
    \mathbb{E}H_1 & =r_1\sum_{c=0}^{\infty}\left(\lambda_1\alpha_{1,c-1}^{(1)}+\lambda_2\alpha_{2,c}^{(1)}\right)+r_2\sum_{c=0}^{\infty}\left(\lambda_1\alpha_{1,c-1}^{(1)}+\lambda_2\alpha_{2,c-1}^{(1)}\right) \nonumber\\
      &=(r_1+r_2)\sum_{c=-1}^{\infty}\left(\lambda_1\alpha_{1,c}^{(1)}+\lambda_2\alpha_{2,c}^{(1)}\right) \nonumber\\
      &=r\sum_{c=-1}^{\infty}(\lambda_1, \lambda_2)
      \left(\begin{matrix}
       \alpha_{1,c}^{(1)} \\
       \alpha_{2,c}^{(1)}
        \end{matrix}
      \right)\nonumber\\
      &=r\left[\lambda_1+\sum_{c=0}^{\infty}(\lambda_1, \lambda_2)\left(\begin{matrix}
  \lambda_1\lambda_2\mathbb{E}\theta_1\mathbb{E}\theta_2 &0 \\
  \lambda_1\mathbb{E}\theta_2& 0
   \end{matrix}\right)^{c+1}
   \left(\begin{matrix}
  1 \\
   0
   \end{matrix}\right)\right]\nonumber\\
   &=r\left[\lambda_1+\sum_{c=0}^{\infty}\lambda_1\lambda_2\mathbb{E}\theta_2(1+\lambda_1\mathbb{E}\theta_1)(\lambda_1\lambda_2\mathbb{E}\theta_1\mathbb{E}\theta_2)^c\right]\nonumber\\
   &=\frac{\lambda_1(1-\rho_1)r}{1-\rho}.
\end{align}

Introduce $\widetilde{G}_{L'}(z)$ such that
\begin{equation}\label{3.5-4}
  \widetilde{G}_{L'}(z)=\widetilde{G}_1(\frac{\lambda_H}{\lambda_1}\widetilde{\theta}_H(\lambda_L(1-z))+\frac{\lambda_L}{\lambda_1}z),
\end{equation}
and a random variable $K$ that satisfies
\begin{equation}\label{3.5-1c}
\begin{split}
  \widetilde{K}(z) =&\mathrm{e}^{-r_1\sum\limits_{c=0}^{\infty}\left[\lambda_H(1-\widetilde{B}_{H,c-1}(z))+\lambda_L(1-\widetilde{B}_{L,c-1}(z))+\lambda_2(1-\widetilde{B}_{2,c}(z))\right]}\\
  &\times\mathrm{e}^{-r_2\sum\limits_{c=0}^{\infty}\left[\lambda_H(1-\widetilde{A}_{H,c-1}(z))+\lambda_L(1-\widetilde{B}_{L,c-1}(z))+\lambda_2(1-\widetilde{B}_{2,c-1}(z))\right]},
\end{split}
\end{equation}
where
\begin{align*}
  \widetilde{B}_{H,c}(z) &=\widetilde{\theta}_H(\lambda_L(1-\widetilde{B}_{L,c}(z))+\lambda_2(1-\widetilde{B}_{2,c}(z))), &\widetilde{B}_{H,-1}(z)&=\widetilde{\theta}_H(\lambda_L(1-z)), \\
  \widetilde{B}_{L,c}(z) &=\widetilde{\theta}'_L(\lambda_2(1-\widetilde{B}_{2,c}(z))), &\widetilde{B}_{L,-1}(z)&=z \\
  \widetilde{B}_{2,c}(z) &=\widetilde{\theta}_2(\lambda_H(1-\widetilde{B}_{H,c-1}(z))+\lambda_L(1-\widetilde{B}_{L,c-1}(z))), &\widetilde{B}_{2,-1}(z)&=1.
\end{align*}
With some efforts, we can prove
\begin{equation}\label{3.5-5}
 \widetilde{G}_{L'}(z)=\widetilde{K}(z)\frac{1-r_{11}\widetilde{M}_{11}(\widetilde{\theta}_H(\lambda_L(1-z)),z,1)}{r_{21}}.
\end{equation}
Combining equation \eqref{3.5-1a}, \eqref{3.5-4} with \eqref{3.5-5} leads to
\begin{equation}\label{3.5-6}
 \widetilde{K}(z)=\widetilde{H}_1(\frac{\lambda_H}{\lambda_1}\widetilde{\theta}_H(\lambda_L(1-z))+\frac{\lambda_L}{\lambda_1}z).
\end{equation}
Differentiating \eqref{3.5-6}, we obtain the first moment of $K$:
\begin{equation*}
    \mathbb{E}K=\mathbb{E}H_1(\frac{\lambda_H}{\lambda_1}\mathbb{E}\theta_H\lambda_L+\frac{\lambda_L}{\lambda_1}) =\frac{\lambda_L(1-\rho_{L'})r}{1-\rho}.
\end{equation*}
Applying the Strong Law of Large Numbers for Renewal Reward Processes to the DSA expressions \eqref{3.5-1b} and \eqref{3.5-1c} yields $\frac{H_1}{r}\rightarrow \frac{ \mathbb{E}H_1}{r}$ and $\frac{K}{r}\rightarrow \frac{ \mathbb{E}K}{r}$ when $r$ tends to infinity. Here we just present the conclusions. Interested readers can refer to Theorem 1 in \cite{VanderMei1999} for more details.
\begin{pro}\label{pro3}
When $r\rightarrow \infty$, we have
\begin{align*}
  \frac{H_1}{r}&\overset{a.s.}{\rightarrow}\frac{\lambda_1(1-\rho_1)}{1-\rho}, &\frac{K}{r}&\overset{a.s.}{\rightarrow}\frac{\lambda_L(1-\rho_{L'})}{1-\rho}.
\end{align*}
\end{pro}

\begin{thm}\label{thm5}
For $0\leq p_1,p_2<1$, when the switch-over times tend to infinity, we have
\begin{align}
     \lim_{r\rightarrow \infty}\mathbb{E}[e^{-s\frac{W_H}{r}}]=&(1-\rho_{L'})\frac{(1-\rho)}{s(1-\rho_1)}\left[1-\mathrm{e}^{-\frac{1-\rho_1}{1-\rho}s}\right]+\rho_{L'},\label{3.5-6aa}\\
     \lim_{r\rightarrow \infty}\mathbb{E}[e^{-s\frac{W_L}{r}}]=&\frac{(1-\rho)}{s(1-\rho_{L'})}\left[1-\mathrm{e}^{-\frac{1-\rho_{L'}}{1-\rho}s}\right],\label{3.5-6ab}\\
     \lim_{r\rightarrow \infty}\mathbb{E}[e^{-s\frac{W_2}{r}}]=&\frac{(1-\rho)}{s(1-\rho_2)}\left[1-\mathrm{e}^{-\frac{1-\rho_2}{1-\rho}s}\right].\label{3.5-6ac}
     \end{align}
     Equivalently,
    \begin{align*}
     \lim_{r\rightarrow \infty}Pr(\frac{W_H}{r}\leq t)&=(1-\rho_{L'})Pr(U_H\leq t)+\rho_{L'},\\
     \lim_{r\rightarrow \infty}Pr(\frac{W_L}{r}\leq t)&=Pr(U_L\leq t),\\
     \lim_{r\rightarrow \infty}Pr(\frac{W_2}{r}\leq t)&=Pr(U_2\leq t),
     \end{align*}
     where $U_i$ ($ i=H,L,2$) is a uniformly $[0,u_i]$ distributed random variable with
     \begin{align*}
       u_H &=\frac{1-\rho_1}{1-\rho}, &u_L &=\frac{1-\rho_{L'}}{1-\rho}, &u_H &=\frac{1-\rho_2}{1-\rho}.
     \end{align*}
\end{thm}

\begin{proof}
From Propostion \ref{pro3}, we have
\begin{equation}\label{3.5-6a}
   \lim_{r\rightarrow \infty}\widetilde{H}_1(1-\frac{s}{\lambda_1r})=\lim_{r\rightarrow \infty}\mathbb{E}(1-\frac{s}{\lambda_1r})^{-\frac{\lambda_1r}{s}\frac{H_1}{r}\left(-\frac{s}{\lambda_1}\right)}
      =\mathbb{E}e^{-\frac{H_1}{r}\frac{s}{\lambda_1}}=e^{-\frac{1-\rho_1}{1-\rho}s},
\end{equation}
The same way gives
\begin{equation}\label{3.5-6b}
  \lim_{r\rightarrow \infty}\widetilde{K}(1-\frac{s}{\lambda_Lr})=e^{-\frac{(1-\rho_{L'})}{1-\rho}s}.
\end{equation}
By the expression \eqref{3.4-12} of $\widetilde{W}_H(s)$, we have
     \begin{equation*}\label{3.5-7}
   \begin{split}
     \lim_{r\rightarrow \infty}\mathbb{E}[e^{-s\frac{W_H}{r}}]=&\lim_{r\rightarrow \infty}\frac{\pi_1(1-\rho)\left[1-r_{11}\widetilde{M}_{11}(1-\frac{s}{\lambda_1r},1)-r_{21}\widetilde{G}_1(1-\frac{s}{\lambda_1r})\right]}{\sigma[\frac{s}{r}-\lambda_H(1-\widetilde{B}_H(\frac{s}{r}))]}\\
       &+\lim_{r\rightarrow \infty}\frac{\rho_L(1-\widetilde{B}_L(\frac{s}{r}))}{[\frac{s}{r}-\lambda_H(1-\widetilde{B}_H(\frac{s}{r}))]\mathbb{E}B_L},\\
       =&\lim_{r\rightarrow \infty}\frac{(1-\rho)\left[1-r_{11}\widetilde{M}_{11}(1-\frac{s}{\lambda_1r},1)-r_{21}\widetilde{G}_1(1-\frac{s}{\lambda_1r})\right]}{r_{21}s[1-\lambda_H\frac{1-\widetilde{B}_H(\frac{s}{r})}{\frac{s}{r}}]}\\
       &+\lim_{r\rightarrow \infty}\frac{\rho_L(1-\widetilde{B}_L(\frac{s}{r}))}{[\frac{s}{r}-\lambda_H(1-\widetilde{B}_H(\frac{s}{r}))]\mathbb{E}B_L},\\
       =&\lim_{r\rightarrow \infty}\frac{(1-\rho)\left[1-r_{11}-r_{21}\widetilde{H}_1(1-\frac{s}{\lambda_1r})\right]}{r_{21}s(1-\rho_H)}+\frac{\rho_L}{1-\rho_H}\\
       =&\frac{(1-\rho)}{s(1-\rho_H)}\left[1-e^{-\frac{1-\rho_1}{1-\rho}s}\right]+\frac{\rho_L}{1-\rho_H}.
   \end{split}
 \end{equation*}
where the second equation follows from $\frac{\pi_1}{\sigma}=\frac{1}{r_{21}r}$, the third equation follows from equation \eqref{3.5-1a} and the last equation results from equation \eqref{3.5-6a}, which is equivalent to \eqref{3.5-6aa}.

 By combining \eqref{3.4-12}, \eqref{3.4-13}, \eqref{3.5-6a} and \eqref{3.5-6b}, \eqref{3.5-6ab} and \eqref{3.5-6ac} can be proved along the same methodology. I won't go into much details here.
\end{proof}
\subsection{Asymptotics for nondeterministic switch-over times under heavy traffic}\label{subsec5.2}
The aim of the present subsection is to study the asymptotic delay for general switch-over times under the heavy traffic when the switch-over times tend to infinity, which actually is the heavy-traffic behaviors of \eqref{3.4-15}-\eqref{3.4-17}.
\begin{thm}\label{thm6}
  \begin{align}
     \lim_{r\rightarrow \infty}\lim_{\rho\rightarrow 1}\mathbb{E}[\mathrm{e}^{-s(1-\rho)\frac{W_H}{r}}]=&(1-\hat{\rho}_{L'})\frac{1}{s(1-\hat{\rho}_1)}\left[1-\mathrm{e}^{-(1-\hat{\rho}_1)s}\right]+\hat{\rho}_{L'},\label{3.5-8}\\
     \lim_{r\rightarrow \infty}\lim_{\rho\rightarrow 1}\mathbb{E}[\mathrm{e}^{-s(1-\rho)\frac{W_L}{r}}]=&\frac{1-\hat{\rho}_H}{s(1-\hat{\rho}_{L'})}\left[1-\mathrm{e}^{-\frac{(1-\hat{\rho}_{L'})}{1-\hat{\rho}_H}s}\right],\label{3.5-9}\\
     \lim_{r\rightarrow \infty}\lim_{\rho\rightarrow 1}\mathbb{E}[\mathrm{e}^{-s(1-\rho)\frac{W_2}{r}}]=&\frac{1}{s(1-\hat{\rho}_2)}\left[1-\mathrm{e}^{-(1-\hat{\rho}_2)s}\right].\label{3.5-10}
     \end{align}
     Equivalently,
    \begin{align*}
     \lim_{r\rightarrow \infty}\lim_{\rho\rightarrow 1}Pr(\frac{(1-\rho)W_H}{r}\leq t)&=(1-\hat{\rho}_{L'})Pr(U_{H1}\leq t)+\hat{\rho}_{L'},\\
     \lim_{r\rightarrow \infty}\lim_{\rho\rightarrow 1}Pr(\frac{(1-\rho)W_L}{r}\leq t)&=Pr(U_{L1}\leq t),\\
     \lim_{r\rightarrow \infty}\lim_{\rho\rightarrow 1}Pr(\frac{(1-\rho)W_2}{r}\leq t)&=Pr(U_{21}\leq t),
     \end{align*}
     where $U_{i1}$ ($i=H,L,2$) is a uniformly $[0,u_{i1}]$ distributed random variable with
     \begin{align*}
       u_{H1} &=1-\hat{\rho}_1, &u_{L1} &=\frac{1-\hat{\rho}_{L'}}{1-\hat{\rho}_H}, &u_{21}&=1-\hat{\rho}_2.
     \end{align*}
\end{thm}
\begin{proof}
Since equations \eqref{3.5-8}-\eqref{3.5-10} follow the same arguments, for compactness, we only present the derivation of \eqref{3.5-8}. By \eqref{3.4-15}, we have
\begin{equation*}
  \begin{split}
   \lim_{r\rightarrow \infty}\lim_{\rho\rightarrow 1}\mathbb{E}[\mathrm{e}^{-s(1-\rho)\frac{W_H}{r}}]&=\lim_{r\rightarrow \infty}(1-\hat{\rho}_{L'})\frac{1}{s(1-\hat{\rho}_1)}\left[1-\left(\frac{\nu_1}{\nu_1+\frac{s}{r}}\right)^\alpha\right]+\hat{\rho}_{L'}\\
      &=\lim_{r\rightarrow \infty}(1-\hat{\rho}_{L'})\frac{1}{s(1-\hat{\rho}_1)}\left[1-\left(1+\frac{s}{\nu_1r}\right)^{-\frac{\nu_1r}{s}\frac{\alpha_1s}{\nu_1}}\right]+\hat{\rho}_{L'}\\
      &=(1-\hat{\rho}_{L'})\frac{1}{s(1-\hat{\rho}_1)}\left[1-\mathrm{e}^{-\frac{\alpha_1s}{\nu_1}}\right]+\hat{\rho}_{L'}\\
      &=(1-\hat{\rho}_{L'})\frac{1}{s(1-\hat{\rho}_1)}\left[1-\mathrm{e}^{-(1-\hat{\rho}_1)s}\right]+\hat{\rho}_{L'},
  \end{split}
\end{equation*}
where $\alpha_1= \frac{2\hat{\rho}_1\hat{\rho}_2}{\hat{\lambda}_H\mathbb{E}B_H^2+\hat{\lambda}_L\mathbb{E}B_L^2+\hat{\lambda}_2\mathbb{E}B_2^2}$ and hence, $\alpha=\alpha_1 r$ and $\frac{\alpha_1}{\nu_1}=\hat{\rho}_2=1-\hat{\rho}_1$.
 \end{proof}
\begin{rmk}
  Compared to the non-priority model, Theorem \ref{thm4}-\ref{thm6} illustrate the following interesting phenomenons:
   \begin{enumerate}
     \item The scaled delay of a high-priority customer is a modified distribution of the scaled delay of a type-$1$ customer in the non-priority model.
     \item The scaled delay of a low-priority customer has the same distribution as the scaled delay of a type-$1$ customer in the non-priority model except the scale parameter.
     \item The scaled delay of a type-2 customer remains the same as that in the non-priority model.
   \end{enumerate}
\end{rmk}
\section{Numerical examples}\label{sec6}
In this section we test the validity of the asymptotic results in Theorem \ref{thm4}-\ref{thm6}, mainly by simulating the polling system with different $\rho$ and $r$, and then comparing the empirical Cumulative Distribution Function (CDF) of the scaled delay and the asymptotic delay. For simplicity, we consider a polling system with exponentially distributed service time and switch-over times.

\begin{table}[htbp]
\centering
 \caption{\label{tab.1}Parameter values in the test of the heavy-traffic behaviors}
 \begin{tabular}{lc}
  \toprule
  Parameter& Considered parameter values \\
  \midrule
  Traffic load & $\rho\in\{0.5, 0.8, 0.9, 0.95, 0.99\}$ \\
   Mean service time & $B=0.85$\\
  Probability of transitions & $p_1=0.4$, $p_2=0.3$\\
  Mean switch-over times & $\sigma=2.4$\\
  \bottomrule
 \end{tabular}
\end{table}
First we consider the polling system in the heavy traffic. The parameter values of the tested polling system can be found in Tab. \ref{tab.1}. We use Matlab to undertake simulations and each simulation runs until at least 2000 customers are served. We take down the waiting times of each customer and plot the CDF of each type of customer in Fig.\ref{fig.1}. The asymptotic distribution of the scaled delay is plotted in the same graph. Actually, the asymptotic distribution of the scaled delay is a residual life of a Gamma distribution or modified Gamma distribution, and hence, can be plotted by the Matlab function ``gamcdf". From Fig.\ref{fig.1}, it is readily showed that the approximation performs very well when $\rho$ is close to $1$.
 \begin{figure}[htp]
\centering
\includegraphics[width=4.7 in]{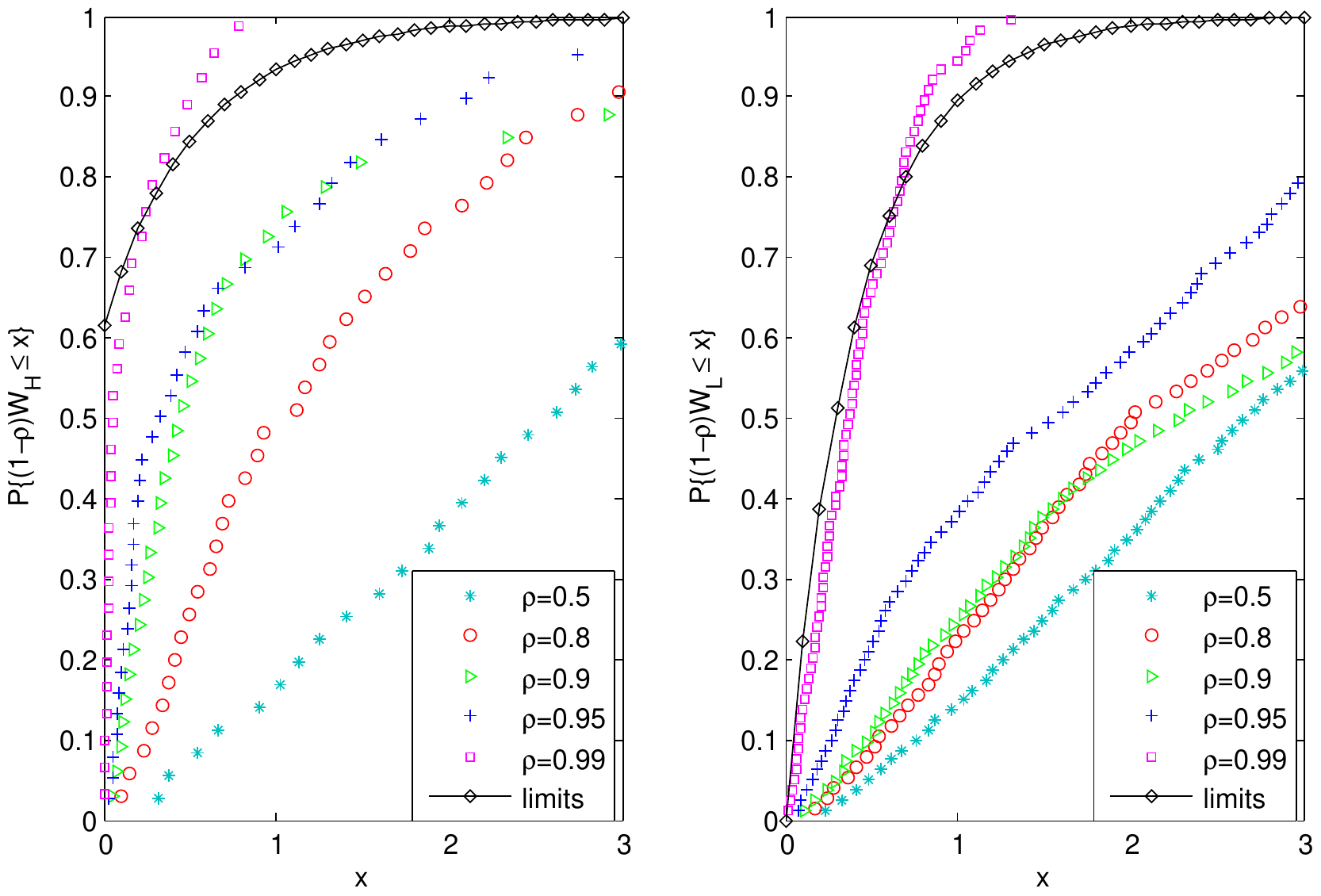} 
\caption {\small The CDFs of scaled delay of type-$H$ and type-$L$ customers for different load}
\label{fig.1}
\end{figure}

\begin{table}[htbp]
\centering
 \caption{\label{tab.2}Parameter values in the test under large deterministic switch-over times}
 \begin{tabular}{lc}
  \toprule
  Parameter& Considered parameter values \\
  \midrule
  Traffic load & $\rho=0.8$ \\
   Mean service time & $B=0.85$\\
  Probability of transitions & $p_1=0.4$, $p_2=0.3$\\
  Mean switch-over times & $r\in \{1, 10, 50, 100, 500\}$\\
  \bottomrule
 \end{tabular}
\end{table}
The scaled delay with large deterministic switch-over times converges to a uniform distribution, which is showed in Fig.\ref{fig.2} and the parameter values of the investigated polling system is given in Tab.\ref{tab.2}. With the same polling simulation procedure, we take down the waiting times of at least 5000 customers and plot their CDFs with Matlab function ``cdfplot". In Fig. \ref{fig.2}, the endpoints of the CDFs in the simulation converge to the limits while other points haven't converged completely. With regard to the running time and the tolerance error, the convergence trend can not be denied in Fig.\ref{fig.2}.
\begin{figure}[htp]
\centering
\includegraphics[width=4.65 in]{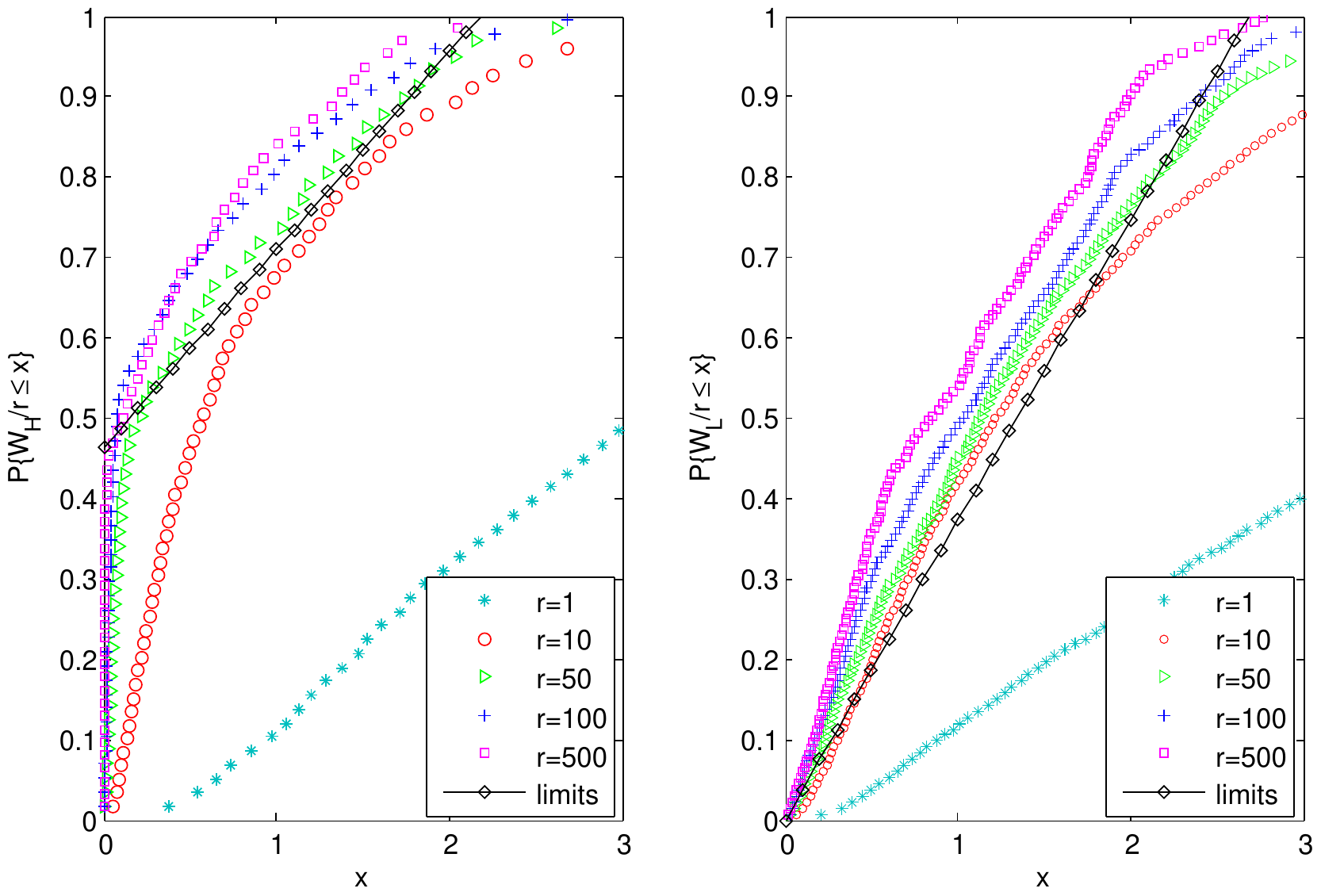} 
\caption {\small The CDFs of scaled delay of type-$H$ and type-$L$ customers under different deterministic switch-over times}
\label{fig.2}
\end{figure}


As for the case of general switch-over times, we consider the convergence process $\lim_{r\rightarrow \infty}\lim_{\rho\rightarrow 1}\frac{(1-\rho)W_i}{r}$. We undertake a simulation work with $\rho=0.99$ and exponentially distributed switch-over times with the same means as the case of the deterministic switch-over times. 40000 customers are served in the simulation procedure. However, the convergence process performs not well. One reason is for the superposition of two limits which enlarges the error of simulation, and the other is that the utilization of two times of scale operations leads to especially very long time required for the simulation to run. To test the validity of Theorem \ref{thm6}, the key is to exploit a new and effective simulation method.

 \section{Conclusions}\label{sec7}
In this paper, we have introduced the priority policy into a $2$-queue Markovian polling system and presented some performance measures like the joint number of customers and waiting times. Moreover, with the DSA, we also investigated the limits of the scaled delays in the case of the heavy-traffic and large switch-over times. Although we have made some achievements of the Markovian priority polling system, it leaves many extension works. Here we just list some topics for further research.
\begin{enumerate}
  \item \emph{Mixed gated/exhaustive service} \ In this case, we can follow the argument in \cite{Boon2009}.
  \item \emph{Multiple priority levels}\  Here we can follow the argument in \cite{Boon2010a}. As for the heavy traffic behaviors and asymptotics with large switch-over times, it is an interesting subject to exploit that whether there are some beautiful conclusions using the DSA.
  \item \emph{Simulation method under large general switch-over times}\ Obviously, the simulation method in our paper can not work effectively for this case. No such simulation work has been published. However, this simulation can help us to exploit the asymptotics of delays with renewal arrivals and to study the limits only with large switch-over times and not in the heavy-traffic.
\end{enumerate}

\end{document}